\documentclass[oneside,a4paper,12pt]{article}
\usepackage[english]{babel}
\usepackage[utf8]{inputenc}
\usepackage[T1]{fontenc}
\usepackage[top=25mm, bottom=25mm, left=25mm, right=25mm]{geometry}
\usepackage{booktabs}
\usepackage{color}
\usepackage{hyperref}
\usepackage{graphicx}
\usepackage{amsmath}
\usepackage{amssymb}
\usepackage{amsthm}
\usepackage{amsfonts}
\usepackage{float}
\graphicspath{{./Figuras/}}    
\definecolor{shadecolor}{rgb}{0.8,0.8,0.8}

\newcommand{\tr}{\mbox{tr}}
\newcommand{\spa}{\mbox{span}}
\newcommand{\Hess}{\mbox{Hess}}
\newcommand{\End}{\mbox{End}}
\newcommand{\grad}{\mbox{grad}}
\newcommand{\Id}{\mbox{Id}}

\newtheorem{theorem}{Theorem}[section]
\newtheorem{proposition}{Proposition}[section]
\newtheorem{corollary}{Corollary}[section]

\newtheorem{lemma}{Lemma}[section]
\newtheorem{Remark}{Remark}[section]

\newtheorem{example}{Example}[section]

\begin{document}
	
	\title{Normally flat submanifolds with semi-parallel Moebius second fundamental form}
	\maketitle
	\begin{center}
		\author{M. S. R. Antas}  
		\footnote{Corresponding author}
		\footnote{The author is supported by CAPES Grant 88887.133756/2025-00.\\
			Data availability statement: Not applicable.}
	\end{center}
	\date{}
	
	\begin{abstract}
	In the Moebius geometry there are two important tensors associated to an umbilic-free immersion $f:M^{n}\to \mathbb{R}^{m}$ named the Moebius metric $\langle \cdot, \cdot \rangle^{*}$ and the Moebius second fundamental form $\beta$. In \cite{semi} was introduced the class of umbilic-free Moebius semi-parallel submanifolds of the unit sphere, which means that $\bar{R}\cdot \beta=0$, where $\bar{R}$ is the van der Waerden-Bortolotti curvature operator associated to $\langle \cdot, \cdot \rangle^{*}$. In this paper, we classify umbilic-free isometric immersions $f:M^{n}\to \mathbb{R}^{m}$ with semi-parallel Moebius second fundamental form and flat normal bundle.
	\end{abstract}
	
	\vspace{2ex}
	
	
	\date{}
	\maketitle

\section{Introduction}

Semi-parallel submanifolds of space forms have been extensively studied over the past decades. (See \cite{lumiste} and the references therein.) In particular, a local classification of semi-parallel submanifolds with flat normal bundle of space forms was obtained in \cite{Dillen}. Recall that an isometric immersion $f:M^{n}\to \mathbb{Q}^{m}_{c}$ of a Riemannian manifold of dimension $n$ into a complete simply connected space form of constant sectional curvature $c$ and dimension $m$ is called \emph{semi-parallel} if its second fundamental form $\alpha:TM^{n}\times TM^{n}\to T^{\perp}M$ with values in the normal bundle satisfies $$\bar{R}(X,Y)\cdot \alpha:=([\bar{\nabla}_{X}, \bar{\nabla}_Y]-\bar{\nabla}_{[X,Y]})\cdot \alpha=0,$$ for all $X,Y \in TM^{n}.$ In particular, this is the case if $\alpha$ is \emph{parallel}, in the sense that $\bar{\nabla}\alpha=0.$ Here $\bar{\nabla}$ stands for the van der Waerden-Bortolotti connection of $f$, given by $$(\bar{\nabla}_{X}\alpha)(Y,Z)=\nabla^{\perp}_{X}\alpha(Y,Z)-\alpha(\nabla_{X}Y,Z)-\alpha(Y, \nabla_{X}Z)$$ where $\nabla$ and $\nabla^{\perp}$ denote the Levi-Civita connection of $M^{n}$ and the normal connection of $f,$ respectively.

Semi-parallel submanifolds are, intrinsically, semi-symmetric Riemannian manifolds, that is, their curvature tensor $R$ satisfies $$R(X,Y)\cdot R:=([\nabla_X, \nabla_Y]-\nabla_{[X,Y]})\cdot R=0$$ for all $X,Y \in TM^{n}$, which is the integrability condition for the equation $\nabla R=0$ that characterizes locally symmetric Riemannian manifolds.

In the context of Moebius geometry of submanifolds, Wang introduced in \cite{Wang} two important Moebius invariant tensors named the \emph{Moebius metric} $\langle \cdot, \cdot\rangle^{*}$ and the \emph{Moebius second fundamental form} $\beta:TM^{n}\times TM^{n}\to T^{\perp}M$ associated to an umbilic-free submanifold $f:M^{n}\to \mathbb{R}^{n+p}$, and proved that, for $p=1$, the pair $(\langle\cdot, \cdot \rangle^{*}, \beta)$ completely determines $f$ up to Moebius transformations in $\mathbb{R}^{n+1}.$ A fundamental theorem for higher codimension is stated in Theorem 9.22 of \cite{DT}. More precisely, let $f:M^{n}\to \mathbb{R}^{m}$ be an isometric immersion of a Riemannian manifold $(M^{n},\langle \cdot, \cdot \rangle)$ into Euclidean space. Let $||\alpha||^{2}\in C^{\infty}(M)$ be given at any point $x \in M^{n}$ by $$||\alpha(x)||^{2}=\sum_{i,j=1}^{n}||\alpha(x)(X_i,X_j)||^{2},$$ where $\{X_i\}_{1 \le i \le n}$ is an orthonormal basis of $T_{x}M.$  Define $\rho\in C^{\infty}(M)$ by $$\rho^{2}=\frac{n}{n-1}(||\alpha||^{2}-n||\mathcal{H}||^{2}),$$ where $\mathcal{H}$ is the mean curvature vector field of $f$. Notice that $\rho$ vanishes precisely at the umbilical points of $f$. The Moebius metric determined by $f$ is defined on the open subset of non-umbilical points of $f$ by $$\langle \cdot, \cdot\rangle^{*}=\rho^{2}\langle \cdot, \cdot \rangle.$$ The Moebius second fundamental form $\beta$ is defined by $\beta=\rho(\alpha-\mathcal{H}\langle \cdot, \cdot \rangle)$.

Thanks to the conformal diffeomorphisms $$\sigma:\mathbb{R}^{m}\to \mathbb{S}^{m}-(-1,\bar{0}) \quad \mbox{and} \quad \tau: \mathbb{H}^{m}\to \mathbb{S}^{m}_{+}$$ defined by \begin{align*}\sigma(u)&=\left(\frac{1-|u|^{2}}{1+|u|^{2}},\frac{2u}{1+|u|^{2}}\right), \quad u \in \mathbb{R}^{m}\\
\tau(y_0,\bar{y}_1)&=\left(\frac{1}{y_0}, \frac{\bar{y}_1}{y_0}\right), \quad (y_0,\bar{y}_1)\in \mathbb{H}^{m},
\end{align*} submanifolds in $\mathbb{R}^{m}$ and $\mathbb{H}^{m}$ can be regarded as one of $\mathbb{S}^{m}$. It
should be noted that due to conformal invariance, the Moebius geometry of submanifolds
is essentially the same whether it is considered in $\mathbb{S}^{m}, \mathbb{R}^{m}$ or $\mathbb{H}^{m}.$ Here $\mathbb{H}^{m}$ is the $m$-dimensional hyperbolic space defined by $\mathbb{H}^{m}=\{(y_0,\bar{y}_1) \in \mathbb{R}^{+}\times \mathbb{R}^{m}: -y_0^{2}+||\bar{y}_1||^{2}=-1\}$ and $\mathbb{S}^{m}_{+}$ is the open hemisphere in $\mathbb{S}^{m}$ whose the first coordinate is positive. From now one our results are stated on the Euclidean space.

Wang's work motivated several authors to investigate umbilic-free immersions whose associated Moebius invariants have a simple structure or particular natural properties (see, among others, \cite{A, AT, classif-moebius-hyper, moebius-ricci-curvature, isopara-three-principals, mobius-homogeneous, CMF-three-principals, Moebius-Laguerre, deform, two-isoparametric, Rodrigues-Keti, blaschke-semi}.)

Recently, Hu, Xie and Zhai introduced in \cite{semi} the class of \emph{Moebius semi-parallel isometric immersions} $f:M^{n}\to \mathbb{R}^{m}$, meaning that,  $\bar{R}\cdot \beta=0$ where $\bar{R}$ is the \emph{van der Waerden-Bortolotti curvature operator} $\bar{R}$ \emph{acting on} $\beta$ defined by $$(\bar{R}\cdot \beta)(X,Y,Z,W):=R^{\perp}(X,Y)\beta(Z,W)-\beta(R^{*}(X,Y)Z,W)-\beta(Z,R^{*}(X,Y)W),$$ for all $X,Y,Z,W \in \mathfrak{X}(M)$. Here $R^{*}$ and $R^{\perp}$ denote the curvature tensors of the connections $\nabla^{*}$ and $\nabla^{\perp}$, respectively, where $\nabla^{*}$ is the Levi-Civita connection associated to $\langle \cdot, \cdot \rangle^{*}.$ Defining the second covariant derivative of $\beta$ by \begin{align*}
	(\nabla^{2}\beta)(X,Y,Z,W)&=\nabla^{\perp}_{X}\left[(\nabla^{\perp}_{Y}\beta)(Z,W)\right]-(\nabla^{\perp}_{\nabla^{*}_{X}Y}\beta)(Z,W)-(\nabla^{\perp}_{Y}\beta)(\nabla^{*}_{X}Z,W)\\
    &-(\nabla^{\perp}_{Y}\beta)(Z, \nabla^{*}_{X}W),
\end{align*} one can show that $\bar{R}\cdot \beta=0$ if and only if \begin{equation*}\label{caracterization-semi}(\nabla^{2}\beta)(X,Y,Z,W)-(\nabla^{2}\beta)(Y,X,Z,W)=0, \quad \forall\,\, X,Y,Z,W \in \mathfrak{X}(M).\end{equation*} In particular, all \emph{Moebius parallel submanifolds} $f:M^{n}\to \mathbb{R}^{m}$, i.e. $(\nabla^{\perp}_{X}\beta)(Y,Z)=0$ for all $X,Y,Z\in \mathfrak{X}(M)$, are semi-parallel, where $$(\nabla^{\perp}_{X}\beta) (X,Y):=\nabla^{\perp}_{X}\beta(Y,Z)-\beta(\nabla^{*}_{X}Y,Z)-\beta(Y, \nabla^{*}_{X}Z).$$ Therefore, the definition of Moebius semi-parallel submanifolds is a natural generalization of the concept of Moebius parallel submanifolds. The problem of classifying Moebius parallel submanifolds of the unit sphere was solved in \cite{classification-mobius-parallel-hyper, Zhai2, Zhai}. In the paper \cite{blaschke-semi}, the authors initiate the study of umbilic-free hypersurfaces in the unit sphere $\mathbb{S}^{n+1}$ with semi-parallel Blaschke tensor.

The authors in \cite{semi} proposed the problem of classifying all Moebius semi-parallel submanifolds of the unit sphere, and, as a result, classified hypersurfaces $f:M^{n}\to \mathbb{S}^{n+1}$, $n \geq 2$, with semi-parallel Moebius second fundamental form and two distinct principal curvatures. They also proved that an umbilic-free hypersurface $f:M^{n}\to \mathbb{S}^{n+1}$, $n \geq 3$, can admit at most three distinct principal curvatures. When $n=3$ and $f$ has three distinct principal curvatures, they proved that the hypersurface must be Moebius parallel. Recently, Antas and Manfio in \cite{AM} and Li and Xie in \cite{li} independently proved that a Moebius semi-parallel hypersurface of dimension $n \geq 4$ with three distinct principal curvatures must have parallel Moebius second fundamental form.  In this paper, we investigate the problem for normally flat submanifolds, that is, we classify umbilic-free Moebius semi-parallel submanifolds $f:M^{n}\to \mathbb{R}^{n+p}$ with flat normal bundle.

The following theorems are the main results of this paper.

\begin{theorem}\label{classfication-two}
	Let $f:M^{n}\to \mathbb{R}^{n+p}$, $n \geq 3$, be an umbilic-free isometric immersion with semi-parallel Moebius second fundamental form and flat normal bundle. If $f$ has two distinct principal normal vector fields, then $f(M^{n})$ is Moebius equivalent to an open subset of one of the following submanifolds:

\begin{description}
	\item[(i)] The image of $\sigma^{-1}$ of the isoparametric torus $ \mathbb{S}^{k}(r)\times \mathbb{S}^{n-k}(\sqrt{1-r^{2}})\subset \mathbb{S}^{n+1}.$
\item[(ii)]The standard cylinder $\mathbb{S}^{k}\times \mathbb{R}^{n-k}$ in $\mathbb{R}^{n+1}$. 
\item[(iii)] The image of $\sigma^{-1}\circ \tau$ of the hyperbolic cylinder $\mathbb{S}^{k}(r)\times \mathbb{H}^{n-k}(\sqrt{1+r^{2}})\subset \mathbb{H}^{n+1}$.
	\item[(iv)] the cylinder over a curve $\gamma:I \to \mathbb{R}^{p+1}$ with the first curvature function given by $\kappa(s)=be^{as}$, where $b\in \mathbb{R}_{+}$ and $a \in \mathbb{R}$ are constants.
    \item[(v)] the generalized cone over a curve $\gamma:I\to \mathbb{S}^{p+1}$  with the first curvature function given by $\kappa(s)=be^{as}$, where $b \in \mathbb{R}_{+}$ and $a \in \mathbb{R}$ are constants.
    \item[(vi)] the rotational submanifold over a curve $\gamma:I\to \mathbb{H}^{p+1}$ with the first curvature function given by $\kappa(s)=\frac{1}{\sqrt{cs+b}}$, where $b,c \in \mathbb{R}$ are non-zero constants. 
\end{description}
\end{theorem}

\begin{theorem}\label{generic}
    Let $f:M^{n}\to \mathbb{R}^{n+p}$, $n \geq 2$ and $p \geq 2$, be an umbilic-free isometric immersion with semi-parallel Moebius second fundamental form and flat normal bundle. If $f$ has $n$ distinct principal normal vector fields then $f$ has vanishing Moebius sectional curvature.
\end{theorem}

Due to Lemma \ref{condition}, all Moebius flat umbilic-free submanifolds in the Euclidean space with flat normal bundle are also Moebius semi-parallel. The preceding theorem for $p=1$ involve only the cases $n=2$ and $n=3$ which was classified in Theorems 5.1 and 5.3 of \cite{semi}, because a Moebius semi-parallel hypersurface has at most three distinct principal curvatures.

We recall that a normal vector field $\bar{\eta}\in \Gamma(N_fM)$ is said to be parallel if $\nabla^{\perp}_{X}\bar{\eta}=0,$ for all $X \in \mathfrak{X}(M).$ When the Moebius principal normal vector fields (see \eqref{moebius-principal}) associated to an umbilic-free submanifold with flat normal bundle are parallel we obtain the following answer to the problem. 

\begin{theorem}\label{semi-constant}
    Let $f:M^{n}\to \mathbb{R}^{n+p}$, $n \geq 4$, be an umbilic-free isometric immersion with semi-parallel Moebius second fundamental form and flat normal bundle. If $f$ has at least three distinct parallel Moebius principal normal vector fields with at least one of them of multiplicity greater than one, then, up to Moebius transformations, $f(M^{n})$ is locally a submanifold with parallel Moebius second fundamental form.
\end{theorem}

When the Moebius principal normal vector fields of $f$ are not necessarily parallel along the normal connection, we obtain that $f$ is locally a submanifold with constant Moebius scalar curvature.

\begin{theorem}\label{three-normals} Let $f:M^{n}\to \mathbb{R}^{n+p}$, $n \geq 4$, be an isometric immersion with flat normal bundle that has at least three distinct principal normal vector fields with at least one of them with multiplicity greater than one. If the Moebius second fundamental form is semi-parallel then $f(M^n)$ is locally a submanifold with constant Moebius scalar curvature.
\end{theorem}

From Theorem 5.5 of \cite{semi}, we have that a Moebius semi-parallel hypersurface $f:M^{n}\to \mathbb{R}^{n+1}$, $n \geq 4,$ with three distinct principal curvatures and constant Moebius scalar curvature has parallel Moebius second fundamental form. Therefore, from Theorem \ref{three-normals} and the main theorem of \cite{classification-mobius-parallel-hyper} we obtain the following consequence.

\begin{corollary}[\cite{AM},\cite{li}]
    If $f:M^{n}\to \mathbb{R}^{n+1}$, $n \geq 4$, is a hypersurface with semi-parallel Moebius second fundamental form and $r \geq 3$ distinct principal curvatures, then, up to Moebius transformations, $f(M^{n})$ is locally a hypersurface with parallel Moebius second fundamental form which is a cone over the torus $\mathbb{S}^{p}(r)\times \mathbb{S}^{q}(\sqrt{1-r^{2}}) \subset \mathbb{S}^{p+q+1}$ where $p,q \in \mathbb{N}$ and $p+q <n$.
\end{corollary}

This paper is organized as follows. In section \ref{preliminaries}, we recall some basic results and definitions from \cite{semi} and \cite{Wang}, and then we prove Theorem \ref{generic}. In section \ref{examples}, we show examples of Moebius semi-parallel submanifolds with flat normal bundle and two distinct principal normal vector fields, then in section \ref{first-thm} we prove Theorem \ref{classfication-two}. In section \ref{second-thm}, we prove Theorem \ref{semi-constant}. In section \ref{third-thm}, we prove Theorem \ref{three-normals}.

\section{Preliminaries}\label{preliminaries}

Following the seminal paper \cite{Wang}, we recall the compatibility equations associated to an umbilic-free submanifold $f:M^{n}\to \mathbb{R}^{n+p}$.

The \emph{Blaschke tensor} $\psi=\psi^{f}$ of $f$ is the symmetric $C^{\infty}(M)$-bilinear form given by \begin{equation}\label{blaschke}\psi(X,Y)=\frac{1}{\rho}\langle \beta(X,Y), \mathcal{H}\rangle+\frac{1}{2\rho^{2}}(||\grad^{*}\rho||^{2}_{*}+||\mathcal{H}||^{2})\langle X,Y\rangle^{*}-\frac{1}{\rho}\Hess^{*}\rho(X,Y)\end{equation} and its \emph{Moebius form} $\omega=\omega^{f}$ is the normal bundle valued one-form defined by $$\omega(X)=-\frac{1}{\rho}(\nabla^{\perp}_{X}\mathcal{H}^{f}+\beta(X,\grad^{*}\rho))$$ where $\grad^{*}$ and $\Hess^{*}$ denote the gradient and the Hessian on $(M^{n}, \langle \,\,,\,\rangle^{*})$.

The Moebius second fundamental form, the Blaschke tensor and the Moebius form of $f$ are Moebius invariant tensors that satisfy the following compatibility equations.

\begin{proposition}[\cite{Wang}]\label{prop1}
	\textit{The following equations hold:}
	
	\begin{description}
		\item[(i)] \textit{The conformal Gauss equation}: \begin{align}\label{G}
			\langle R^{*}(X,Y)Z,W\rangle^{*}=&\langle \beta(X,W), \beta(Y,Z)\rangle - \langle \beta(X,Z), \beta(Y,W)\rangle\nonumber\\ 
			&+\psi(X,W)\langle Y,Z \rangle^{*}+\psi(Y,Z)\langle X,W\rangle^{*}\nonumber\\
			&-\psi(X,Z)\langle Y,W \rangle^{*}-\psi(Y,W)\langle X,Z \rangle^{*}
		\end{align} \textit{for all} $X,Y,Z,W \in \mathfrak{X}(M)$.
		\item[(ii)] \textit{The conformal Codazzi equations}:
		\begin{equation}\label{C1}
			({^{f}\nabla^{\perp}_{X}}\beta)(Y,Z)-({^{f}\nabla^{\perp}_{Y}}\beta)(X,Z)=\omega((X \wedge Y)Z)
		\end{equation} \textit{and} \begin{equation}\label{C2}
			(\nabla^{*}_{X}\psi)(Y,Z)-(\nabla^{*}_{Y}\psi)(X,Z)=\langle \omega(Y), \beta(X,Z)\rangle-\langle \omega(X), \beta(Y,Z)\rangle
		\end{equation} \textit{for all} $ X,Y,Z \in \mathfrak{X}(M),$ \textit{where} $$ ({^{f}\nabla^{\perp}_{X}}\beta)(Y,Z)={^{f}\nabla^{\perp}_{X}}\beta(Y,Z)-\beta(\nabla^{*}_{X}Y,Z)-\beta(Y, \nabla^{*}_{X}Z),$$  $$(\nabla^{*}_{X}\psi)(Y,Z)=X(\psi(Y,Z))-\psi(\nabla^{*}_{X}Y, Z)-\psi(Y, \nabla^{*}_{X}Z)$$ \textit{and}
		$$(X \wedge Y)Z=\langle Y,Z\rangle^{*} X-\langle X,Z\rangle^{*} Y.$$ 
		\item[(iii)] \textit{The conformal Ricci equations}: \begin{equation}\label{R1}
			d\omega(X,Y)=\beta(Y, \hat{\psi}X)-\beta(X, \hat{\psi}Y)
		\end{equation} \textit{and} \begin{equation}\label{R2}
			\langle R^{\perp}(X,Y)\xi, \eta \rangle=\langle [B_{\xi}, B_{\eta}]X,Y\rangle^{*}
		\end{equation} \textit{for all} $ X,Y \in \mathfrak{X}(M)$ \textit{and} $\xi, \eta \in \Gamma(N_{f}M)$, \textit{with} $\hat{\psi}, B_{\xi} \in \Gamma(\End(TM))$ \textit{given by} \begin{equation*}
			\langle \hat{\psi}X,Y\rangle^{*}=\psi(X,Y) \quad \mbox{and} \quad \langle B_{\xi}X,Y\rangle^{*}=\langle \beta(X,Y), \xi \rangle.
		\end{equation*}
	\end{description}	
\end{proposition}

We point out that the Moebius form of an umbilic-free Moebius semi-parallel isometric immersion $f:M^{n}\to \mathbb{R}^{n+p}$, $n \geq 3$, with flat normal bundle is closed. (See Proposition 3.2 of \cite{semi}). Therefore, by the conformal Ricci equation, there exists an orthonormal frame $X_1, \ldots,X_n$ with respect to $\langle \,\,,\,\rangle^{*}$ that diagonalizes $\beta$ and $\psi$ such that
$$\beta(X_i,X_i)=\bar{\eta}_i \quad \mbox{and} \quad \psi(X_i,X_i) =\theta_i, \quad \forall\,\, 1 \le i \le n.$$

Let $f:M^{n}\to \mathbb{R}^{n+p}$ be a proper isometric immersion with flat normal bundle and let ${\eta}_1, \ldots, {\eta}_k$, $k \geq 2$, be the distinct principal normal vector fields of $f$ whose associated eigenbundles are denoted by $E_{\eta_1}, \ldots, E_{\eta_k}.$ Given $X\in \Gamma(E_{\eta_i})$ and $Y \in \mathfrak{X}(M)$, we have \begin{align}
    \beta(X,Y)&=\rho(\alpha(X,Y)-\langle X,Y\rangle \mathcal{H})\nonumber\\
    &=\rho(\langle X,Y\rangle \eta_i-\langle X,Y\rangle \mathcal{H})\nonumber\\
    &=\langle X,Y\rangle^{*}\rho^{-1}(\eta_i-\mathcal{H}).\label{moebius-principal}
\end{align}

The normal vector fields $\bar{\eta}_i=:\rho^{-1}(\eta_i-\mathcal{H})$, $1 \le i \le k,$ are called \emph{Moebius principal normal vector fields} of $f.$ Notice that $E_{\eta_i}$ is the smooth distribution associated with $\bar{\eta}_i$ and the number of distinct Moebius principal normal vector fields is $k$.

The next result is a simple characterization of Moebius semi-parallel submanifolds with flat normal bundle.

\begin{lemma}\label{condition}
	Let $f:M^{n}\to \mathbb{R}^{m}$, $n \geq 2$, be an umbilic-free isometric immersion with flat normal bundle. Then $f$ is Moebius semi-parallel if and only if \begin{equation}\label{condition-semi}
		\langle \bar{\eta}_i, \bar{\eta}_j\rangle+\theta_i+\theta_j=0, \quad 1 \le i \ne j \le k,
	\end{equation} where $\bar{\eta}_1, \ldots, \bar{\eta}_k$ are the distinct Moebius principal normal vector fields of $f$ and $\theta_1, \ldots, \theta_k$ are the eigenvalues of $\psi$.
\end{lemma}

\begin{proof}
    $f$ is Moebius semi-parallel if and only if $\beta(R^{*}(X_i,X_j)X_u, X_l)+\beta(X_u, R^{*}(X_i,X_j)X_l)=0$ for all $1 \le i,j,u,l \le n$, that is, $\langle R^{*}(X_i,X_j)X_u, X_l\rangle^{*}(\bar{\eta}_l-\bar{\eta}_u)= 0.$ Choosing $j=u$ and $i=l$ with $1 \le i \ne j \le n$, we get $(\langle R^{*}(X_i,X_j)X_j,X_i\rangle^{*})(\bar{\eta}_i-\bar{\eta}_j)=0.$ Then, \begin{equation}\label{curvature-condition}K^{*}(X_i,X_j)=0 \quad \mbox{or} \quad \bar{\eta}_i=\bar{\eta}_j \quad \forall \,\, 1 \le i \ne j \le n.\end{equation}	
 By the conformal Gauss equation, we have	$\langle \bar{\eta}_i, \bar{\eta}_j\rangle +\theta_i+\theta_j=0$ for all $1 \le i \ne j \le k$.
\end{proof}

The Lemma \ref{condition} shows that the number of distinct eigenvalues of the Blaschke tensor is at most $k$ and all Moebius flat umbilic-free submanifolds in the Euclidean space with flat normal bundle are also Moebius semi-parallel.

The following two results are a consequence of \eqref{condition-semi} and the conformal Gauss equation.

\begin{corollary}
Any umbilic-free surface $f:M^{2}\to \mathbb{R}^{m}$ with flat normal bundle is Moebius semi-parallel if and only if its Moebius sectional curvatures vanishes.
\end{corollary}

\begin{corollary}
	Let $f:M^{n}\to \mathbb{R}^{m}$, $n \geq 2$, be an umbilic-free Moebius semi-parallel isometric immersion with flat normal bundle. If $f$ has constant Moebius sectional curvature $c$, then $c=0$.
\end{corollary}

The classification of submanifolds $f:M^{n}\to \mathbb{R}^{n+p}$, $n \geq 5$ and $2p \le n$, with flat normal bundle and vanishing Moebius sectional curvature was done in \cite{AT}.

\begin{proof}[{Proof of Theorem \ref{generic}}:]
Let $\bar{\eta}_1, \ldots, \bar{\eta}_n$ be the distinct Moebius principal normal vector fields of $f$ and let $\theta_1, \ldots, \theta_n$ be the eigenvalues of the Blaschke tensor. By Lemma \ref{condition} and the conformal Gauss equation, we have $K^{*}(X_i,X_j)=0$ for all $ 1 \le i \ne j \le n$.
\end{proof}

\section{Examples}\label{examples}

In this section, we recall the definition of some standard examples of submanifolds with flat normal bundle and classify the one with semi-parallelity of the Moebius second fundamental form.

\begin{example}\label{cylinder}
	The cylinder in $\mathbb{R}^{n+p}$ over a curve $\gamma:I \to \mathbb{R}^{p+1}$ is defined by $$f=\gamma \times \Id:I \times \mathbb{R}^{n-1}\to \mathbb{R}^{n+p},$$ where $\Id:\mathbb{R}^{n-1}\to \mathbb{R}^{n-1}$ is the identity map.

    The Moebius metric determined by $f$ is given by $$\langle \,\, , \, \rangle^{*}=\kappa^{2}(ds^{2}+du_2^{2}+\ldots+du_n^{2}),$$ where $s$ is the arc-lenght parameter of $\gamma$, $\kappa(s)$ is the first curvature function of $\gamma$ and $(u_2, \ldots, u_n)$ are canonical coordinates on $\mathbb{R}^{n-1}.$

    The Moebius second fundamental form and the Blaschke tensor of $f$ are written in the orthonormal basis $\{\frac{1}{\kappa}\frac{d}{ds}, \frac{1}{\kappa}\frac{d}{du_2}, \ldots, \frac{1}{\kappa}\frac{d}{du_n}\}$ as \begin{align*}
		\beta\left(\kappa^{-1} \frac{d}{ds}, \kappa^{-1}\frac{d}{ds}\right)&=\frac{n-1}{n}\xi \quad \mbox{and} \quad
		\beta \left(\kappa^{-1}\frac{d}{du_i}, \kappa^{-1}\frac{d}{du_i}\right)=-\frac{1}{n}\xi,\\
        \psi\left(\kappa^{-1}\frac{d}{ds}, \kappa^{-1}\frac{d}{ds}\right)&=-\frac{\kappa_{ss}}{\kappa^{3}}+\frac{3}{2}\frac{k_s^{2}}{\kappa^{4}}+\frac{2n-1}{2n^{2}} \quad \mbox{and} \quad 
        \psi\left(\kappa^{-1}\frac{d}{du_i}, \kappa^{-1}\frac{d}{du_i}\right)=-\frac{1}{2}\left(\frac{\kappa_s^{2}}{\kappa^{4}}+\frac{1}{n^{2}}\right),
	\end{align*} for all $2 \le i \le n$, where $\xi$ is a unit normal vector field of $f$.
\end{example}    

\begin{example}\label{cone1}
	The generalized cone in $\mathbb{R}^{n+p}$ over a curve $\gamma:I \to \mathbb{S}^{p+1}$ is the immersion $f:I \times \mathbb{H}^{n-1}\to  \mathbb{R}^{n+p}$ given by $$f(s, z)=\Theta \circ (\gamma, \Id)(s,z)=(z_1 \gamma(s), z_2,\ldots,z_{n-1}),$$ where $\mathbb{H}^{n-1}=\mathbb{R}_{+}\times \mathbb{R}^{n-2}$ is endowed with the hyperbolic metric $dz^{2}=\frac{1}{z_1^{2}}(dz_1^{2}+\ldots+dz_{n-1}^{2}),$ $Id:\mathbb{H}^{n-1}\to \mathbb{H}^{n-1}$ is the identity map and $\Theta: \mathbb{S}^{p+1}\times \mathbb{H}^{n-1}\to \mathbb{R}^{n+p}$ is the conformal diffeomorphism defined by $\Theta(y,z)=(z_1y, z_2, \ldots, z_{n-1}).$ 

    The Moebius metric determined by $f$ is given by $$\langle \,\,,\,\rangle^{*}=\kappa^{2}(ds^{2}+dz^{2}),$$ where $s$ is the arc-lenght parameter of $\gamma$ and $\kappa(s)$ is the first curvature function of $\gamma.$

    The Moebius second fundamental form and the Blaschke tensor of $f$ are written in the orthonormal basis $\{\kappa^{-1}\frac{d}{ds}, \kappa^{-1}z_1\frac{d}{dz_1}, \ldots, \kappa^{-1}z_1\frac{d}{dz_{n-1}}\}$ as 

    \begin{align*}
        \beta \left(\kappa^{-1}\frac{d}{ds}, \kappa^{-1}\frac{d}{ds}\right)&=\frac{n-1}{n}\xi \quad \mbox{and} \quad 
		\beta \left(\kappa^{-1}z_1\frac{d}{dz_i}, \kappa^{-1}z_1\frac{d}{dz_i}\right)=-\frac{1}{n}\xi,\\
        \psi \left(\kappa^{-1} \frac{d}{ds}, \kappa^{-1}\frac{d}{ds}\right)&=-\frac{\kappa_{ss}}{\kappa^{3}}+\frac{3}{2}\frac{\kappa_s^{2}}{\kappa^{4}}+\frac{1}{2\kappa^{2}}+\frac{2n-1}{2n^{2}},\\
        \psi \left(\kappa^{-1}z_1 \frac{d}{dz_i}, \kappa^{-1}z_1 \frac{d}{dz_i}\right)&=-\frac{1}{2}\left(\frac{\kappa_s^{2}}{\kappa^{4}}+\frac{1}{\kappa^{2}}+\frac{1}{n^{2}}\right),
    \end{align*} for all $1 \le i \le n-1$, where $\xi$ is a unit normal vector field of $f.$
\end{example}

\begin{example}\label{rotacao}
	The rotational submanifold over a curve $\gamma:I \to \mathbb{H}^{p+1}$ is the immersion $f:I \times \mathbb{S}^{n-1} \to \mathbb{R}^{n+p}$ defined by $$f(s,y)=\Theta \circ (\gamma, \Id)(s,y)=(z_1(s),\ldots, z_{p}(s), z_{p+1}(s)y),$$ where $\mathbb{H}^{p+1}=\{(z_1, \ldots, z_{p+1}): z_{p+1}>0\}$ is endowed with the hyperbolic metric $ds^{2}=\frac{1}{z_{p+1}^{2}}(dz_1^{2}+\ldots+dz_{p+1}^{2}),$ $\Id:\mathbb{S}^{n-1}\to \mathbb{S}^{n-1}$ is the identity map and $\Theta: \mathbb{H}^{p+1}\times \mathbb{S}^{n-1} \to \mathbb{R}^{n+p}$ is the conformal diffeomorphism $\Theta(z,y)=(z_1, \ldots, z_p, z_{p+1}y).$

    The Moebius metric determined by $f$ is $$\langle \,\,,\,\rangle^{*}=\kappa^{2}(\gamma^{*}ds^{2}+dy^{2}|_{\mathbb{S}^{n-1}}),$$ where $s$ is the arc-lenght parameter of $\gamma$ and $\kappa(s)$ is the first curvature function of $\gamma$.

    The Moebius second fundamental form and the Blaschke tensor of $f$ are written in the orthonormal basis $\{\kappa^{-1}\frac{d}{ds}, \kappa^{-1}e_1, \ldots, \kappa^{-1}e_{n-1}\}$ as \begin{align*}
	\beta \left(\kappa^{-1}\frac{d}{ds}, \kappa^{-1}\frac{d}{ds}\right)&=\frac{n-1}{n}\Theta_{*}\frac{\xi}{z_{p+1}} \quad \mbox{and} \quad 
	\beta \left(\kappa^{-1}e_i, \kappa^{-1}e_i\right)=-\frac{1}{n}\Theta_{*}\frac{\xi}{z_{p+1}},\\
    \psi \left(\kappa^{-1}\frac{d}{ds}, \kappa^{-1}\frac{d}{ds}\right)&=\frac{\kappa_{ss}}{\kappa^{3}}-\frac{5}{2}\frac{(\kappa_s)^{2}}{\kappa^{4}}-\frac{1}{2\kappa^{2}}+\frac{2n-1}{2n^{2}},\\
	\psi \left(\kappa^{-1}e_i, \kappa^{-1}e_i\right)&=-\frac{1}{2}\left[\frac{(\kappa_s)^{2}}{\kappa^{4}}-\frac{1}{\kappa^{2}}+\frac{1}{n^{2}}\right],
\end{align*} for all $1 \le i \le n-1$, where $e_1, \ldots, e_{n-1}$ is an orthonormal frame with respect to $dy^{2}|_{\mathbb{S}^{n-1}}$ and $\xi$ is a unit normal vector field of $\gamma.$
 \end{example}

The next lemma, classify the Moebius semi-parallel submanifold \begin{equation}\label{sub-examples}
    f=\Theta \circ (\gamma, \Id):I\times \mathbb{Q}^{n-1}_{c}\to \mathbb{R}^{n+p},
\end{equation} where $\Theta$ is as in Example \ref{cylinder}, \ref{cone1} and \ref{rotacao}, according to $c=0,$ $c=-1$ or $c=1$, respectively. 
 
\begin{lemma}\label{redutible-examples}
    Let $\gamma:I \to \mathbb{Q}^{p+1}_{-c}$ be a curve parametrized by arc length. Then the submanifold \eqref{sub-examples} is Moebius semi-parallel if and only if the first curvature function $\kappa(s)$ of $\gamma$ is given by
    \begin{equation}
      \kappa(s)=\begin{cases}
            b_1e^{b_2s}, \quad \mbox{if} \,\,\, c=0 \,\,\, \mbox{or} \,\,\, c=-1,\\
            \frac{1}{\sqrt{b_3s+b_4}}, \quad \mbox{if} \,\,\, c=1,
        \end{cases}
    \end{equation}
where $b_2,b_3, b_4\in \mathbb{R}$ and $b_1 \in \mathbb{R}_{+}$ are arbitrary constants.
\end{lemma} 

\begin{proof}
     It is sufficient to notice that from Lemma \ref{condition} and the Moebius invariants computed in Examples \ref{cylinder}, \ref{cone1} and \ref{rotacao}, we see that $f$ is Moebius semi-parallel if and only if \begin{equation}\label{edo}\begin{cases}
		\frac{\kappa_{ss}}{\kappa^{3}}-\frac{\kappa_s^{2}}{\kappa^{4}}=0, \quad \mbox{if}\,\, {c}=0 \,\,\mbox{or}\,\, {c}=-1,\\
		 \frac{\kappa_{ss}}{\kappa^{3}}-3\frac{\kappa_s^{2}}{\kappa^{4}}=0, \quad \mbox{if}\,\, {c}=1.
	\end{cases}\end{equation} Then the result follows.
\end{proof}

\section{Proof of Theorem \ref{classfication-two}} \label{first-thm}

In order to prove Theorem \ref{classfication-two} we need the next lemma, which is a generalization of Proposition 4.2 of \cite{AT}.
\begin{lemma}\label{class-two}
	Let $f:M^{n}\to \mathbb{R}^{n+p}$, $n \geq 3$ and $p \geq 1$, be an isometric immersion with flat normal bundle that has two distinct principal normal vector fields $\eta_1$ and $\eta_2$ of multiplicities $k$ and $n-k$, for $1 \le k \le n-2$, respectively. If $\omega^{f}$ is closed then $f(M^{n})$ is Moebius equivalent to an open subset of a cylinder, a generalized cone or a rotational submanifold over an isometric immersion $g:M^{k}\to \mathbb{Q}^{p+k}_{c}$, where $c=0,1$ or $-1$, respectively. In particular, if $k \geq 2$ then $g$ is umbilical.
\end{lemma}

\begin{proof}
    Since $\omega^{f}$ is closed and $f$ has two distinct principal normal vector fields $\eta_1$ and $\eta_2$ of multiplicities $k$ and $n-k$, respectively, the conformal Ricci equation implies that there exists an orthonormal frame $X_1, \ldots,X_n$ with respect to the Moebius metric that diagonalizes simultaneously the Moebius second fundamental form $\beta$ and the Blaschke tensor $\psi$ such that $$\beta(X_j,X_j)=\bar{\eta}_1 \quad \mbox{and} \quad \beta(X_i,X_i)=\bar{\eta}_2$$ for all $1 \le j \le k$ and $k+1 \le i \le n$.

	Consider the smooth distributions $$\Delta=\spa\{X_i:k+1 \le i \le n\} \quad \mbox{and} \quad \Delta^{\perp}=\spa\{X_i: 1 \le i \le k\}.$$ Since $\beta$ is traceless and $||\beta||_{*}^{2}=\frac{n-1}{n},$ we obtain $$k \bar{\eta}_1+(n-k)\bar{\eta}_2=0 \quad \mbox{and} \quad k||\bar{\eta}_1||^{2}+(n-k)||\bar{\eta}_2||^{2}=\frac{n-1}{n}.$$ Then $$||\bar{\eta}_1||=\sqrt{\frac{(n-1)(n-k)}{kn^{2}}} \quad \mbox{and} \quad ||\bar{\eta}_2||=\sqrt{\frac{k(n-1)}{n^{2}(n-k)}}.$$ In particular, $$\bar{\eta}_1=\sqrt{\frac{(n-1)(n-k)}{kn^{2}}}\xi_1 \quad \mbox{and} \quad \bar{\eta}_2=-\sqrt{\frac{k(n-1)}{n^{2}(n-k)}}\xi_1,$$ for some unit vector field $\xi_1\in \Gamma(N_fM).$

	The conformal Codazzi equation  
	$$(\nabla^{\perp}_{X_j}\beta)(X_i,X_i)-(\nabla^{\perp}_{X_i}\beta)(X_j,X_i)=\omega((X_j\wedge X_i)X_i)$$ for all $k+1\le i \ne j \le n$, implies $$ \omega(X_j)=\nabla^{\perp}_{X_j}\bar{\eta}_2.$$

	On the other hand, the equation $$(\nabla^{\perp}_{X_j}\beta)(X_i,X_l)-(\nabla^{\perp}_{X_i}\beta)(X_j,X_l)=\omega((X_j \wedge X_i)X_l), $$ for all $1 \le i,l \le k$ and $k+1 \le j \le n$, is equivalent to $$\langle X_i,X_l\rangle^{*}\omega(X_j)=\langle X_i,X_l\rangle^{*}\nabla^{\perp}_{X_j}\bar{\eta}_1+\langle \nabla^{*}_{X_i}X_l,X_j\rangle^{*}(\bar{\eta}_2-\bar{\eta_1}). $$ 	
	 Since $\bar{\eta}_1=\frac{k-n}{k}\bar{\eta}_2$, we have \begin{equation*}
		\langle X_i,X_l \rangle^{*}{\nabla^{\perp}_{X_j}\bar{\eta}_2}=\langle \nabla^{*}_{X_i}X_l, X_j \rangle^{*}\bar{\eta}_2.
	\end{equation*} Taking the inner product of the preceding equation with $\bar{\eta}_2$ we get \begin{equation*}
		\langle \nabla^{*}_{X_i}X_l, X_j \rangle^{*}=0,		\end{equation*} for all $1 \le i, l \le k$ and $k+1 \le j \le n.$ Therefore, $\Delta^{\perp}$ is totally geodesic with respect to the Moebius metric.

	Using the relation between the Riemannian 
	connections of the metrics  $\langle \,,\,\rangle_{f}$ and $\langle \,,\,\rangle^{*}$, $$\nabla^{*}_{X_i}X_l={^{f}\nabla_{X_i}}X_l+\frac{1}{\rho}(X_i(\rho )X_l+X_l(\rho)X_i-\langle X_i, X_l \rangle_{f}\,\grad_{f}\rho),$$ we obtain \begin{equation*}
		\langle {^{f}\nabla_{X_i}}X_l, X_j \rangle_{f}=\langle X_i, X_l\rangle_{f} \langle \grad_{f}\log \rho, X_j \rangle_{f}
	\end{equation*} for all $1 \le i, l \le k$ and $k+1 \le j \le n.$ 	Thus, $\Delta^{\perp}$ is an umbilic distribution with respect to the induced metric by $f$ whose mean curvature vector field is $$ (\grad_{f}\log \rho)_{\Delta}.$$

	Now we show that $\Delta^{\perp}$ is a spherical distribution with respect to the induced metric by $f$, that is, 
	\begin{equation}\label{spherical}
		\langle {^{f}\nabla_{X_i}}(\grad_{f} \log \rho )_{\Delta}, X_j \rangle_{f}=0, \quad \forall \,\,\, 1 \le i \le k \quad \mbox{and} \quad k+1 \le j \le n.
	\end{equation}

	It follows from the definition of the Blaschke tensor (see \eqref{blaschke}) and the fact that $X_1, \ldots, X_n$ diagonalizes $\beta$ and $\psi$ that $$ \Hess^{*}\rho(X_i,X_j)=0, \quad \forall \,\,\, 1 \le i \le k \quad \mbox{and} \quad k+1 \le j \le n.$$ 	Since \begin{equation*}
		\nabla^{*}_{X_i}\grad^{*}\rho= {^{f}\nabla_{X_i}}\grad^{*}\rho+\frac{1}{\rho}((\grad^{*}\rho)(\rho)X_i+ X_i(\rho)\grad^{*}\rho-\langle X_i, \grad^{*}\rho \rangle_{f}\grad_{f}\rho)
	\end{equation*} and $
	\grad^{*}\rho=\dfrac{\grad_{f}\rho}{\rho^{2}},
	$ we get \begin{equation*}
		0=  \Hess^{*}\rho(X_i,X_j)=\langle \nabla^{*}_{X_i}\grad^{*}\rho, X_j \rangle^{*}=\rho^{2}\langle {^{f}\nabla_{X_i}}\grad^{*}\rho, X_j \rangle_{f}.\end{equation*}
	However, \begin{equation*}
		\langle {^{f}\nabla_{X_i}}\grad^{*}\rho, X_j \rangle_{f}=\langle X_i(\rho^{-1})\grad_{f}\log \rho+\rho^{-1}{^{f}\nabla_{X_i}}\grad_{f}\log \rho, X_j \rangle_{f},
	\end{equation*} then \begin{equation}\label{esfera1}
		\langle ^{f}\nabla_{X_i}\grad_{f}\log \rho, X_j \rangle_{f}=X_i(\log \rho)X_j(\log \rho).\end{equation}
	On the other hand, using 
	$$\nabla^{*}_{X_i}X_j={^{f}\nabla_{X_i}}X_j+\frac{1}{\rho}(X_i(\rho)X_j+X_j(\rho)X_i)$$
	and  $\nabla^{*}_{X_i}X_j \in \Delta$, we have
	\begin{align}
		\langle ^{f}\nabla_{X_i}(\grad_{f}\log \rho)_{\Delta^{\perp}},X_j \rangle_{f}&=-\langle (\grad_{f}\log \rho)_{\Delta^{\perp}}, {^{f}\nabla_{X_i}}X_j \rangle_{f}\nonumber\\
		&=X_i(\log \rho)X_j(\log \rho). \label{esfera2}
	\end{align}
	Therefore, \eqref{spherical} follows from \eqref{esfera1} and \eqref{esfera2}.

	Since $\Delta=E_{\eta_2}$, the conclusion follows from the main theorem of \cite{umbilicity}.	
\end{proof}

\begin{proof}[Completion of the proof of Theorem \ref{classfication-two}:] In Proposition 3.2 of \cite{semi} was proved that $f$ has closed Moebius form. Then by Lemma \ref{class-two}, $f(M^{n})$ is locally Moebius equivalent to a cylinder, a generalized cone or a rotational submanifold over an isometric immersion $g:M^{k}\to \mathbb{Q}^{p+k}_c$, where $c=0,1$ or $-1$, respectively. When $k \geq 2$, these examples are stated in 
items \textbf{(i)}, \textbf{(ii)} and \textbf{(iii)} and they are examples of Moebius parallel hypersurfaces (see the main theorem of \cite{classification-mobius-parallel-hyper}). When $k=1$, the conclusion follows from Lemma \ref{redutible-examples}.
\end{proof}

\section{Proof of Theorem \ref{semi-constant}}\label{second-thm}
Let $f:M^{n}\to \mathbb{R}^{n+p}$, $n \geq 4$, be a proper isometric immersion with flat normal bundle and semi-parallel Moebius second fundamental form. Let $\bar{\eta}_1, \ldots, \bar{\eta}_k$, $k \geq 3$, be the distinct Moebius principal normal vector fields of $f$ of multiplicities $m_1, \ldots, m_k$ and assume that at least one of them has multiplicity greater than one, say, $m_k\geq 2.$ Since $d\omega=0,$ there exists an orthonormal frame $X_1, \ldots, X_n$ with respect to the Moebius metric such that for each $j \in \{1, \ldots, k\}$
$$\beta(X_i,X_i)=\bar{\eta}_j \quad \mbox{and} \quad \psi(X_i,X_i)=\theta_j,$$ for all $m_1+\ldots+m_{j-1}+1 \le i \le m_1+\ldots +m_j.$ Therefore, the eigenbundle distribution associated to $\eta_j$ is $E_{\eta_j}=\spa\{X_k\}_{k=m_1+\ldots+m_{j-1}+1}^{m_1+\ldots+m_j}$.

Assume $\bar{\eta}_i$ is parallel for all $1 \le i \le k$. Since $\langle \bar{\eta}_i, \bar{\eta}_j\rangle+\theta_i+\theta_j=0$ for all $1 \le i \ne j \le k$ and $k \geq 3$, we can easily get that $\theta_i$ is a constant function for all $1 \le i \le k.$ Due to the formula $\tr \, \psi = \frac{n^{2}s^{*}+1}{2n}$, where $s^{*}$ denotes the scalar curvature of the Moebius metric, we conclude that $s^{*}$ is a constant function.

\begin{proposition}\label{null-1}
   If $m_i \geq 2$ then $\omega(X_i)=0,$ for all $X_i \in E_{\eta_i}$.
\end{proposition}

\begin{proof}
From equation $$(\nabla^{\perp}_{X_k}\beta)(X_j,X_j)-(\nabla^{\perp}_{X_j}\beta)(X_k,X_j)=\omega((X_k \wedge X_j)X_j),$$ for all $X_k,X_j \in E_{\eta_i}$, $k \ne j$, we obtain  $\omega(X_k)=0$.
\end{proof}

\begin{Remark}
    If $m_i \geq 2$ for all $1 \le i \le k$, then $\omega=0,$ and, therefore, according to the definition given in \cite{A}, $f$ is a Moebius isoparametric submanifold.
\end{Remark}

\begin{lemma}\label{possibilities}
At any point $p \in M^{n}$, $\Vert\Bar{\eta}_i\Vert^{2}+2\theta_i$ can be zero for at most one $i \in \{1, \ldots,k\}$.
\end{lemma}
\begin{proof}
    Let $i \ne j \in \{1, \ldots,k\}$, then $\Vert\Bar{\eta}_i-\Bar{\eta}_j\Vert^{2}=\Vert\Bar{\eta}_i\Vert^{2}-2\langle \Bar{\eta}_i, \Bar{\eta}_j\rangle+\Vert\Bar{\eta}_j\Vert^{2}.$ Since $\langle \Bar{\eta}_i, \Bar{\eta}_j\rangle+\theta_i+\theta_j=0$ we have \begin{align}
        \Vert\Bar{\eta}_i-\Bar{\eta}_j\Vert^{2}&=\Vert\Bar{\eta}_i\Vert^{2}-2(-\theta_i-\theta_j)+\Vert\Bar{\eta}_j\Vert^{2}\nonumber\\
        &=(\Vert\Bar{\eta}_i\Vert^{2}+2\theta_i)+(\Vert\Bar{\eta}_j\Vert^{2}+2\theta_j).\label{differ-norm}
    \end{align} If $\Vert\Bar{\eta}_i\Vert^{2}+2\theta_i=0=\Vert\Bar{\eta}_j\Vert^{2}+2\theta_j$ then $\Bar{\eta}_i=\Bar{\eta}_j,$ which is a contradiction.
\end{proof}

We have two possibilities on an open subset $U$ of $M^{n}$:

\begin{description}
    \item[(a)] $\Vert\bar{\eta}_k\Vert^{2}+2\theta_k=0.$ 
    \item[(b)] $\Vert\bar{\eta}_k\Vert^{2}+2\theta_k \ne 0$.
\end{description}

Assume that case \textbf{(b)} occurs.

\begin{proposition}\label{null-2}
    $\omega(X_i)=0$ for all $X_i \in E_{\eta_k}^{\perp}.$
\end{proposition}

\begin{proof}
The equation $$(\nabla^{\perp}_{X_i}\beta)(X_j,X_j)-(\nabla^{\perp}_{X_j}\beta)(X_i,X_j)=\omega((X_i\wedge X_j)X_j)$$ for $X_i \in E_{\eta_k}^{\perp}$ and $X_j \in E_{\eta_k}$, yields \begin{equation}\label{cd111}
    \omega(X_i)=\langle \nabla^{*}_{X_j}X_j,X_i\rangle^{*}(\bar{\eta}_i-\bar{\eta}_k). 
\end{equation} 
On the other hand, the equation $$(\nabla^{*}_{X_i}\psi)(X_j,X_j)-(\nabla^{*}_{X_j}\psi)(X_i,X_j)=-\langle \omega(X_i), \beta(X_j,X_j)$$ yields \begin{equation}\label{cd12}
    -\langle \omega(X_i), \bar{\eta}_k\rangle =\langle \nabla^{*}_{X_j}X_j,X_i\rangle^{*}(\theta_i-\theta_k).
\end{equation}
Taking the inner product of \eqref{cd111} with $\bar{\eta}_k$, adding it with \eqref{cd12} and using Lemma \ref{condition} we obtain $$\langle \nabla^{*}_{X_j}X_j,X_i\rangle^{*}(-\Vert\bar{\eta}_k\Vert^{2}-2\theta_k)=0,$$ which implies $\langle \nabla^{*}_{X_j}X_j,X_i\rangle^{*}=0.$ We conclude from \eqref{cd111} that $\omega(X_i)=0$. 
\end{proof}

\begin{Remark}
    If case \textbf{(b)} occurs, we conclude from Propositions \ref{null-1} and \ref{null-2} that $\omega=0.$
\end{Remark}

Assume that case \textbf{(a)} occurs. If at least $\dim E_{\eta_\ell}\geq 2$ for some $1 \le \ell \le k-1$, then Lemma \ref{possibilities} implies that $\Vert\bar{\eta}_\ell\Vert^{2}+2\theta_\ell \ne 0$ on $U.$ In this last case, we apply Propositions \ref{null-1} and \ref{null-2} to $\bar{\eta}_\ell$, from which we can conclude the vanishing of the Moebius form. If $m_i=1$ for all $1 \le i \le k-1$, then $f$ has vanishing Moebius curvature. In order to prove Theorem \ref{semi-constant}, we need to prove that if $f$ is a submanifold with vanishing Moebius curvature then $\omega=0.$

\begin{lemma}\label{orthogonal}
    If $\Vert\bar{\eta}_k\Vert^{2}+2\theta_k=0$ on $U$ then $k \le p+1$ and $$\langle \bar{\eta}_i-\bar{\eta}_k, \bar{\eta}_j-\bar{\eta}_k\rangle=0$$ for all $1 \le i \ne j \le k-1.$
\end{lemma}

\begin{proof}
 Indeed, \begin{align*}
        \langle \bar{\eta}_i-\bar{\eta}_k, \bar{\eta}_j-\bar{\eta}_k\rangle&=\langle \bar{\eta}_i, \bar{\eta}_j\rangle-\langle \bar{\eta}_i, \bar{\eta}_k\rangle-\langle \bar{\eta}_k, \bar{\eta}_j\rangle+\Vert\bar{\eta}_k\Vert^{2}\\
        &=(-\theta_i-\theta_j)-(-\theta_i-\theta_k)-(-\theta_k-\theta_j)+\Vert\bar{\eta}_k\Vert^{2}\\
        &=\Vert\bar{\eta}_k\Vert^{2}+2\theta_k=0,
    \end{align*} for all $1 \le i \ne j  \le  k-1$. Therefore, we have $\{\bar{\eta}_i-\bar{\eta}_k: 1 \le i \le k-1\}$ is a linearly independent subset of $T^{\perp}M$. Thus $k-1 \le p.$
\end{proof}

\begin{lemma}\label{r.o.}
    There exists an orthonormal subset $\{\xi_i\}_{i=1}^{k-1}$ of $\Gamma(T^{\perp}M)$ such that \begin{equation}
        \bar{\eta}_i-\bar{\eta}_k=\sqrt{\Vert\bar{\eta}_i\Vert^{2}+2\theta_i}\xi_i,
    \end{equation} for all $1 \le i \le k-1.$
\end{lemma}

\begin{proof}
As we observe in \eqref{differ-norm}, $\Vert\bar{\eta}_i-\bar{\eta}_k\Vert^{2}=\Vert\bar{\eta}_i\Vert^{2}+2\theta_i$ for all $1 \le i \le k-1$. The conclusion follows from Lemma \ref{orthogonal}. 
\end{proof}

In the sequence, we denote $\sqrt{\Vert\bar{\eta}_i\Vert+2\theta_i}$ by $h_i.$

\begin{proposition}
Let $f:M^{n}\to \mathbb{R}^{n+p}$, $n \geq 4$, be an isometric immersion with vanishing Moebius curvature and flat normal bundle. If $f$ has at least three distinct parallel Moebius principal normal vector fields with at least one of them with multiplicity greater than one then $f$ is Moebius isoparametric, and, hence, $f(M^{n})$ is locally Moebius equivalent to one of the following submanifolds 

$$\mathbb{S}^{1}(r_1)\times \ldots \times \mathbb{S}^{1}(r_{n-\lambda})\times \mathbb{R}^{\lambda} \subset \mathbb{R}^{2n-\lambda}$$ for $2 \le \lambda \le n-2$.
\end{proposition}

\begin{proof}
By the conformal Codazzi equation $$(\nabla^{\perp}_{X_i}\beta)(X_j,X_j)-(\nabla^{\perp}_{X_j}\beta)(X_i,X_j)=\omega((X_i\wedge X_j)X_j),$$ for all $1 \le i \le k-1$ and $k \le j \le n,$ we obtain $$\omega(X_i)=\langle \nabla^{*}_{X_j}X_j,X_i\rangle^{*}(\bar{\eta}_i-\bar{\eta}_k).$$ Using Lemma \ref{r.o.} and denoting by $\delta$ the mean cuvature vector of $E_{\eta_k}$ with respect to $\langle \,\,,\,\rangle^{*}$, we get $$\omega(X_i)=\langle \delta, X_i\rangle^{*}h_i \xi_i.$$  On the other hand, the conformal Codazzi equation $$(\nabla^{\perp}_{X_i}\beta)(X_\ell,X_\ell)-(\nabla^{\perp}_{X_\ell}\beta)(X_i,X_\ell)=\omega((X_i\wedge X_\ell)X_\ell)$$ for all $1 \le i \ne \ell \le k-1$, implies $$\omega(X_i)=\langle \nabla^{*}_{X_\ell}X_\ell, X_i\rangle^{*}(h_i\xi_i-h_\ell \xi_\ell).$$ Therefore, $\langle \nabla^{*}_{X_\ell}X_\ell,X_i\rangle^{*}=\langle \delta, X_i\rangle^{*}=0$. Then $\omega(X_i)=0,$ for all $1 \le i \le k-1.$

By Proposition \ref{null-1}, $\omega|_{E_{\eta_k}}=0.$ Therefore, we have proved that $\omega=0.$ The conclusion follows from Theorem 1.1 of \cite{A}.
\end{proof}

\begin{proof}[Completion of the proof of Theorem \ref{semi-constant}:]
   In order to prove this result, we need to show \begin{equation}\label{parallel}
   (\nabla^{\perp}_{X_i}\beta)(X_j,X_\ell)=0
   \end{equation} for all $1 \le i,j,\ell \le n.$ Indeed, if $j=\ell$, then \eqref{parallel} is trivially satisfied. Since we have proved that $\omega=0$, the conformal Codazzi equation becomes $(\nabla^{\perp}_{X}\beta)(Y,Z)=(\nabla^{\perp}_{Y}\beta)(X,Z)$ for all $X,Y,Z \in \mathfrak{X}(M).$ Assume that $j \ne \ell$, then $(\nabla^{\perp}_{X_i}\beta)(X_j,X_\ell)=(\nabla^{\perp}_{X_j}\beta)(X_i,X_\ell)$ implies \begin{equation}\label{cd11}\langle \nabla^{*}_{X_i}X_j,X_\ell\rangle^{*}(\bar{\eta}_j-\bar{\eta}_\ell)=\langle \nabla^{*}_{X_j}X_i,X_\ell\rangle^{*}(\bar{\eta}_i-\bar{\eta}_\ell).\end{equation} If $\bar{\eta}_i=\bar{\eta}_j \ne \bar{\eta}_\ell$, then $\langle \nabla^{*}_{X_i}X_j,X_\ell\rangle^{*}=0,$ because as proved in Lemma 3.1 of \cite{A}, $E_{\eta_i}$ is totally geodesic with respect to $\langle \,\,,\,\rangle^{*},$ therefore, $(\nabla^{\perp}_{X_i}\beta)(X_j,X_\ell)=0.$ If $\bar{\eta}_i=\bar{\eta}_\ell$, then $(\nabla^{\perp}_{X_i}\beta)(X_j,X_\ell)=0.$ Now, assume that $\bar{\eta}_i \ne \bar{\eta}_j \ne \bar{\eta}_\ell \ne \bar{\eta}_i$, then by the equation $(\nabla^{*}_{X_i}\psi)(X_j,X_\ell)=(\nabla^{*}_{X_j}\psi)(X_i,X_\ell)$ we get \begin{equation}\label{cd222}\langle \nabla^{*}_{X_i}X_j,X_\ell\rangle^{*}(\theta_j-\theta_\ell)=\langle \nabla^{*}_{X_j}X_i,X_\ell \rangle^{*}(\theta_i-\theta_\ell).\end{equation} By Lemma \ref{possibilities}, we can assume that $\Vert\bar{\eta}_j\Vert^{2}+2\theta_j \ne 0$, then taking the inner product of \eqref{cd11} with $\bar{\eta}_j$, adding it with \eqref{cd222} and using Lemma \ref{condition} we obtain $\langle \nabla^{*}_{X_i}X_j,X_\ell\rangle^{*}=0,$ from which we conclude $(\nabla^{\perp}_{X_i}\beta)(X_j,X_\ell)=0.$ This completes the proof of \eqref{parallel}.
\end{proof}

\section{Proof of Theorem \ref{three-normals}}\label{third-thm}

We study proper umbilic-free isometric immersions with flat normal bundle and semi-parallel Moebius second fundamental form provided the Moebius principal normal vector fields are not necessarily parallel along the normal connection of $f.$

Throughout this section, we consider the notation in the first 
paragraph of the preceding section. According to Lemma \ref{possibilities}, we have two cases to consider on an open subset $U$ of $M^{n}$. 

\begin{description}
    \item[(i)] $ \Vert\bar{\eta}_i\Vert^{2}+2\theta_i \ne 0$ for all $1 \le i \le k.$
    \item[(ii)] There is exactly one $1 \le i \le k$ such that $\Vert\bar{\eta}_i\Vert^{2}+2\theta_i =0.$
\end{description}

\textbf{Case (i)}
 We observe that this is the only possibility in the hypersurface case. Indeed, it was proved in \cite{semi} that an umbilic-free hypersurface $f:M^{n}\to \mathbb{R}^{n+1},$ $n \geq 4$, with semi-parallel Moebius second fundamental form can have at most three distinct Moebius principal curvatures. If $f$ has three distinct Moebius principal curvatures $\bar{\lambda}_1, \bar{\lambda}_2$ and $\bar{\lambda}_3$ and for example $\bar{\lambda}_1^{2}+2\theta_1=0,$ then \begin{align*}(\bar{\lambda}_2-\bar{\lambda}_1)(\bar{\lambda}_3-\bar{\lambda}_1)&=\bar{\lambda}_2\bar{\lambda}_3-\bar{\lambda}_2 \bar{\lambda}_1-\bar{\lambda}_1 \bar{\lambda}_3 + \bar{\lambda}_1^{2}\\
&=-\theta_2-\theta_3+\theta_2+\theta_1+\theta_1+\theta_3+\bar{\lambda}_1^{2}=0,
\end{align*} which is a contradiction.

\begin{lemma}\label{geodesic}
    If $m_i \geq 2$ then $X_{\ell}(\Vert\bar{\eta}_i\Vert^{2}+2\theta_i)=0$ for all $X_\ell \in E_{\eta_i}.$
\end{lemma}

\begin{proof}
    By the conformal Codazzi equation 
    $$(\nabla^{*}_{X_\ell}\psi)(X_j,X_j)-(\nabla^{*}_{X_j}\psi)(X_\ell,X_j)=\langle \omega(X_j), \beta(X_\ell,X_j)\rangle-\langle \omega(X_\ell), \beta(X_j,X_j)\rangle$$ for all $X_\ell, X_j \in E_{\eta_i}$, $\ell \ne j$, we obtain \begin{equation}\label{cd2}
        X_\ell(\theta_i)=-\langle \omega(X_\ell), \bar{\eta}_i\rangle.
    \end{equation}
The equation $$(\nabla^{\perp}_{X_\ell}\beta)(X_j,X_j)-(\nabla^{\perp}_{X_j}\beta)(X_\ell,X_j)=\omega((X_\ell\wedge X_j)X_j)$$ yields \begin{equation}\label{cdij}
    \nabla^{\perp}_{X_\ell}\bar{\eta}_i=\omega(X_\ell).
\end{equation} Taking the inner product of \eqref{cdij} with $\bar{\eta}_i$ and adding it with \eqref{cd2} we get  $X_\ell(\Vert\bar{\eta}_i\Vert^{2}+2\theta_i)=0$ for all $X_\ell \in E_{\eta_i}.$\end{proof}

\begin{lemma}\label{umbilical}
    $E_{\eta_i}$ is an umbilical distribution with mean curvature vector field $$-(\grad^{*}(\log |\Vert \bar{\eta}_i \Vert^{2}+2\theta_i|^{-\frac{1}{2}}))_{E_{\eta_i}^{\perp}}$$ for all $1 \le i \le k$.
\end{lemma}

\begin{proof}
    Let $X_u,X_v \in E_{\eta_i}$ and $X_z \in E_{\eta_i}^{\perp}$. The conformal Codazzi equation $$(\nabla^{\perp}_{X_u}\beta)(X_z,X_v)-(\nabla^{\perp}_{X_z}\beta)(X_u,X_v)=\omega((X_u \wedge X_z)X_v)$$ yields \begin{equation}\label{cd3}
        \langle \nabla^{*}_{X_u}X_v,X_z\rangle^{*}(\bar{\eta}_i-\bar{\eta}_z)-\langle X_u,X_v\rangle^{*}\nabla^{\perp}_{X_z}\bar{\eta}_i=-\langle X_u,X_v\rangle^{*}\omega(X_z).
    \end{equation} On the other hand, equation $$(\nabla^{*}_{X_u}\psi)(X_z,X_v)-(\nabla^{*}_{X_z}\psi)(X_u,X_v)=\langle \omega(X_z), \beta(X_u,X_v)\rangle-\langle \omega(X_u), \beta(X_z,X_v)\rangle$$ yields 
    \begin{equation}\label{cd4}
        \langle \nabla^{*}_{X_u}X_v,X_z\rangle^{*}(\theta_i-\theta_z)-\langle X_u,X_v\rangle^{*}X_z(\theta_i)=\langle X_u,X_v\rangle^{*}\langle \omega(X_z), \bar{\eta}_i\rangle.
    \end{equation}

    Taking the inner product of \eqref{cd3} with $\bar{\eta}_i$, adding it with \eqref{cd4} and using the Lemma \ref{condition} we get $$\langle \nabla^{*}_{X_u}X_v,X_z\rangle^{*}(\Vert\bar{\eta}_i\Vert^{2}+2\theta_i)-\langle X_u,X_v\rangle^{*}\frac{X_z(\Vert\bar{\eta}_i\Vert^{2}+2\theta_i)}{2}=0,$$ which is equivalent to $$\langle \nabla^{*}_{X_u}X_v,X_z\rangle^{*}=-\langle X_u,X_v\rangle^{*}\langle \grad^{*}(\log |\Vert\bar{\eta}_i\Vert^{2}+2\theta_i|^{-\frac{1}{2}}),X_z\rangle^{*}.$$\end{proof}

    %\begin{lemma}\label{connulity-geodesic}
     %   $E_{\eta_j}^{\perp}$ is totally geodesic with respect to $\langle \,\,,\,\rangle^{*}$ if and only if $$(\grad^{*}(||\bar{\eta}_i||^{2}+2\theta_i))_{E_{\eta_j}}=0$$ for all $1 \le i \ne j \le k.$
    %\end{lemma}

 %\begin{proof}
  %   Let $X_a \in E_{\eta_i}$, $X_b \in E_{\eta_\ell}$ and $X_c \in E_{\eta_j}$ with $1 \le i \ne j \ne \ell \ne i \le k.$ The conformal Codazzi equations $$(\nabla^{\perp}_{X_a}\beta)(X_b,X_c)=(\nabla^{\perp}_{X_b}\beta)(X_a,X_c) \quad \mbox{and} \quad (\nabla^{*}_{X_a}\psi)(X_b,X_c)=(\nabla^{*}_{X_b}\psi)(X_a,X_c)$$ yield $$\langle \nabla^{*}_{X_a}X_b,X_c\rangle^{*}(\bar{\eta}_\ell-\bar{\eta}_j)=\langle \nabla^{*}_{X_b}X_a,X_c\rangle^{*}(\bar{\eta}_i-\bar{\eta}_j)\,\,\, \mbox{and} \,\,\, \langle \nabla^{*}_{X_a}X_b,X_c\rangle^{*}(\theta_\ell-\theta_j)=\langle \nabla^{*}_{X_b}X_a,X_c\rangle^{*}(\theta_i-\theta_j).$$ Taking the inner product of the first equation with $\bar{\eta}_\ell$, adding it to the second equation and using Lemma \ref{condition} we get $\langle \nabla^{*}_{X_a}X_b,X_c\rangle^{*}=0.$ The conclusion of this result follows from Lemma \ref{umbilical}.
 %\end{proof}   

    \begin{lemma}\label{integrability}
        $E_{\eta_i}^{\perp}$ is an integrable distribution for all $1 \le i \le k.$ 
    \end{lemma}

\begin{proof}
    Let $X_u,X_v \in E_{\eta_i}^{\perp}$ and $X_z \in E_{\eta_i}$. The conformal Codazzi equation $$(\nabla^{\perp}_{X_u}\beta)(X_v,X_z)-(\nabla^{\perp}_{X_v}\beta)(X_u,X_z)=\omega((X_u\wedge X_v)X_z)$$ gives $$\langle \nabla^{*}_{X_u}X_v,X_z\rangle^{*}(\bar{\eta}_v-\bar{\eta}_i)=\langle \nabla^{*}_{X_v}X_u,X_z\rangle^{*}(\bar{\eta}_u-\bar{\eta}_i).$$

    If $\bar{\eta}_v=\bar{\eta}_u,$ then  $[X_u,X_v]\in E_{\eta_i}^{\perp}$. If $\bar{\eta}_v \ne \bar{\eta}_u$, then taking the inner product of the preceding equation with $\bar{\eta}_v$ and using Lemma \ref{condition} we obtain \begin{equation}\label{cd5}
    \langle \nabla^{*}_{X_u}X_v,X_z\rangle^{*}(\Vert\bar{\eta}_v\Vert^{2}+\theta_i+\theta_v)=\langle \nabla^{*}_{X_v}X_u,X_z\rangle^{*}(\theta_i-\theta_u).\end{equation} However, equation $$(\nabla^{*}_{X_u}\psi)(X_v,X_z)=(\nabla^{*}_{X_v}\psi)(X_u,X_z)$$ yields \begin{equation}\label{cd6}
        \langle \nabla^{*}_{X_u}X_v,X_z\rangle^{*}(\theta_v-\theta_i)=\langle \nabla^{*}_{X_v}X_u,X_z\rangle^{*}(\theta_u-\theta_i).
    \end{equation}

    Adding \eqref{cd5} and \eqref{cd6} we have $\langle \nabla^{*}_{X_u}X_v,X_z\rangle^{*}(\Vert\bar{\eta}_v\Vert^{2}+2\theta_v)=0$, which implies $$\langle \nabla^{*}_{X_u}X_v,X_z\rangle^{*}=0.$$ Therefore, $E_{\eta_i}^{\perp}$ is integrable.
\end{proof}

	Recall that an orthogonal net $\mathcal{E}=(E_i)_{i=1}^{k}$ on a Riemannian manifold $M^{n}$ is called a \emph{twisted product net} if $E_i$ is umbilical and $E_i^{\perp}$ is integrable for all $i=1, \ldots, k$. From Lemmas \ref{umbilical} and \ref{integrability} it follows that $\mathcal{E}=\{E_{\eta_1}, \ldots, E_{\eta_k}\}$ is a twisted product net on $(M^{n}, \langle \,\,,\,\rangle^{*})$.

 In Corollary 1 of \cite{decomposition} was obtained the following decomposition theorem  for manifolds endowed with twisted product nets.

\begin{theorem}[\cite{decomposition}]\label{decomposition}
    Let $\mathcal{E}=(E_i)_{i=1}^{k}$ be a TP-net on a Riemannian manifold $M.$ Then for each $x \in M$ there exists a local product representation $\Phi:\prod_{i=1}^{k}M_i \to U$ of $\mathcal{E}$ with $x \in U \subset M$, which is an isometry with respect to a twisted product metric on $\prod_{i=1}^{k}M_i.$
\end{theorem}

By Theorem \ref{decomposition}, at each $x\in M^{n}$ there exists an open subset $U$ of $M^{n}$ containing $x$ and a product representation $\Phi:\prod_{i=1}^{k}M_i \to U$ of $\mathcal{E}$ that is an isometry with respect to a twisted product metric 
\begin{equation}\label{mobius-metric1}
	\langle \,\,,\,\rangle=\sum_{i=1}^{k}\rho_i^{2} \pi_i^{*} \langle\,\,,\,\rangle_{i}
\end{equation} of $\prod_{i=1}^{k}M_i$, for some positive \emph{twisting functions} $\rho_i \in C^{\infty}\left(\prod_{i=1}^{k}M_i\right)$, $1 \le i \le k,$ where $\pi_i:\prod_{j=1}^{k}M_j\to M_i$ denotes the canonical projection.

Let $(E_i)_{i=1}^{k}$ be the product net of $\prod_{i=1}^{k}M_i$, that is, $E_i(x)={\tau_i^{x}}_{*}(T_{x_i}M_i)$ for all $x=(x_1, \ldots,x_k)\in \prod_{i=1}^{k}M_i$, where $\tau_i^{x}:M_i\to \prod_{i=1}^{k}M_i$ is the standard inclusion. As observed in Proposition 2-(c) of \cite{decomposition}, $E_i$ is an umbilical distribution with mean curvature vector field $-(\grad(\log \circ \rho_i))_{E_i^{\perp}}$, where $\grad$ is the gradient with respect to $\langle \,\,,\,\rangle.$

We claim that there exist orthogonal coordinates $(x_1,\ldots,x_{n})$ on $\prod_{i=1}^{k}M_i$ and smooth functions $r_i=r_i(x_{m_1+\ldots m_{i-1}+1}, \ldots, x_{m_1+\ldots+m_i})$ such that $$\rho_i=r_i|\Vert\bar{\eta}_i\Vert^{2}+2\theta_i|^{-\frac{1}{2}} \circ \Phi,$$ for each $1 \le i \le k.$

Indeed, we proved in Lemma \ref{umbilical} that $E_{\eta_i}$ is an umbilical distribution with mean curvature vector field $-(\grad^{*}(\log |\Vert\bar{\eta}_i\Vert^{2}+2\theta_i|^{-\frac{1}{2}}))_{E_{\eta_i}^{\perp}}.$ Since $\Phi$ is a product representation, that is, $\Phi_{*}E_i(x)=E_{\eta_i}(\Phi(x))$, and an isometry, we have that \begin{align*}
    -\langle \grad^{*}(\log |\Vert\bar{\eta}_i\Vert^{2}+2\theta_i|^{-\frac{1}{2}}), X_\ell\rangle^{*}&=\langle \nabla^{*}_{X_a}X_a, X_\ell\rangle^{*}\\
    &=\langle\nabla^{*}_{{\tau_{i}^{x}}_{*}\bar{X}_a}\Phi_{*}{\tau_{i}^{x}}_{*}\bar{X}_a, \Phi_{*}{\tau_{i}^{x}}_{*}\bar{X}_\ell \rangle^{*}\\
    &=\langle\Phi_{*}\nabla_{{\tau_{i}^{x}}_{*}\bar{X}_a}{\tau_{i}^{x}}_{*}\bar{X}_a, \Phi_{*}{\tau_{i}^{x}}_{*}\bar{X}_\ell \rangle^{*}\\
    &=\langle\nabla_{{\tau_{i}^{x}}_{*}\bar{X}_a}{\tau_{i}^{x}}_{*}\bar{X}_a,{\tau_{i}^{x}}_{*}\bar{X}_\ell \rangle\\
    &=-\langle \grad(\log \circ \rho_i), {\tau_{i}^{x}}_{*}\bar{X}_\ell\rangle,
\end{align*} for all $X_a \in E_{\eta_i}$ and $X_\ell \in E_{\eta_i}^{\perp}.$ Then \begin{align*}
(\Phi_{*} \grad(\log \circ \rho_i))_{(E_{\eta_i}(\Phi(x)))^{\perp}}&=(\grad^{*}(\log |\Vert\bar{\eta}_i \Vert^{2}+2\theta_i|^{-\frac{1}{2}}))_{(E_{\eta_i}(\Phi(x)))^{\perp}}\\
&=(\Phi_{*}\grad(\log \circ |\Vert \bar{\eta}_i\Vert^{2}+2\theta_i|^{-\frac{1}{2}}\circ \Phi))_{(E_{\eta_i}(\Phi(x)))^{\perp}}.
\end{align*}
Since $E_{\eta_1}^{\perp}, \ldots, E_{\eta_k}^{\perp}$ are integrable, it follows that $f$ is locally holonomic on $U$, and hence there exist orthogonal coordinates $(x_1, \ldots, x_n)$ on $\prod_{i=1}^{k}M_i$ such that $$E_{\eta_i}=\spa\left\{\frac{\partial}{\partial x_{m_1+\ldots+m_{i-1}+1}}, \ldots, \frac{\partial}{\partial x_{m_1+\ldots+m_i}}\right\}$$ for all $1 \le i \le k.$ Therefore, $\rho_i=r_i|\Vert\bar{\eta}_i\Vert^{2}+2\theta_i|^{-\frac{1}{2}}\circ \Phi$, for some positive smooth functions $r_i=r_i(x_{m_1+\ldots+m_{i-1}+1}, \ldots, x_{m_1+\ldots+m_i}).$ This proves our claim.

Therefore, omitting the isometry $\Phi$, we can write the metric \eqref{mobius-metric1} as the orthogonal metric $$\langle \cdot, \cdot \rangle^{*}=|\Vert\bar{\eta}_1\Vert^{2}+2\theta_1|^{-1}\sum_{j=1}^{m_1}V_j^{2}dx_j^{2}+\ldots+|\Vert\bar{\eta}_k\Vert^{2}+2\theta_k|^{-1}\sum_{j=m_1+\ldots + m_{k-1}+1}^{n}V_j^{2}dx_j^{2},$$ where for each $1 \le i \le k$, $V_j=r_i \sqrt{\left\langle \frac{\partial}{\partial x_j}, \frac{\partial}{\partial x_j}\right\rangle_i}$ for all $\frac{\partial}{\partial x_j}\in E_{\eta_i}.$

For each $1 \le i \le k$, denote by $$v_j=|\Vert\bar{\eta}_i\Vert^{2}+2\theta_i|^{-\frac{1}{2}}V_j,$$ for all $m_1+\ldots+m_{i-1}+1 \le j \le m_1+\ldots+m_i.$ Then, we can write $$\langle \cdot, \cdot \rangle^{*}=\sum_{j=1}^{n}v_j^{2}dx_j^{2}$$ and $\{X_j=v_j^{-1}\frac{\partial}{\partial x_j}: 1 \le j \le n\}$ is an orthonormal frame on $(M^{n}, \langle \cdot, \cdot \rangle^{*})$. By the conformal Gauss equation and the Lemma \ref{condition}, $f$ is Moebius semi-parallel if and only if $$K^{*}(X_a,X_b)=0$$ for all $X_a \in E_{\eta_i}$ and $X_b \in E_{\eta_j}$ with $1 \le i \ne j \le k.$ We also set $h_{ij}=\frac{1}{v_i}\frac{\partial v_j}{\partial x_i}$ for all $1 \le i \ne j \le n.$

Let $E_{\eta_1}, \ldots, E_{\eta_t}$ be line eigenbundle distributions and let $E_{\eta_{t+1}}, \ldots, E_{\eta_k}$ be distributions with dimension at least two. (When $t=0$ then all eigenbundle distributions $E_{\eta_1}, \ldots, E_{\eta_k}$ have dimension at least two.) The Moebius metric can be written as follows: 
$$\langle \cdot, \cdot \rangle^{*}=\sum_{j=1}^{t}|\Vert\bar{\eta}_j\Vert^{2}+2\theta_j|^{-1}dx_j^{2}+ \sum_{i=t+1}^{k}\sum_{j=m_{t+1}+\ldots+m_{i-1}+t+1}^{m_{t+1}+\ldots+m_i+t}|\Vert\bar{\eta}_i\Vert^{2}+2\theta_i|^{-1}V_j^{2}dx_j^{2}.$$

Given arbitrary choices of $1 \le i \le t$ and $t+1 \le r \le k$, fix the coordinates in $M_{\ell}$ for all $1 \le \ell \le k$ with $\ell \ne i,r.$ Then $M_i\times M_r$ is a totally geodesic submanifold of $\prod_{s=1}^{k}M_s$ whose induced metric is $$\langle \cdot, \cdot \rangle^{*}_{ir}=|\Vert\bar{\eta}_i\Vert^{2}+2\theta_i|^{-1}dx_i^{2}+|\Vert\bar{\eta}_r||^{2}+2\theta_r|^{-1}\sum_{j=m_{t+1}+\ldots+m_{r-1}+t+1}^{m_{t+1}+\ldots+m_r+t}V_j^{2}dx_j^{2}.$$

From Lemma \ref{geodesic} it follows that $\langle \cdot, \cdot \rangle^{*}_{ir}$ is a warped product metric with warping function $\mu:=|\Vert\bar{\eta}_r\Vert^{2}+2\theta_r|^{-\frac{1}{2}}.$ Since $f$ is Moebius semi-parallel, then $\Hess\, \mu=0$ where $\Hess$ is computed with respect to $|\Vert\bar{\eta}_i\Vert^{2}+2\theta_i|^{-1}dx_i^{2}.$ Therefore, $\frac{\partial h_{ij}}{\partial x_i}=0.$ However, $$h_{ij}=-\frac{1}{2}V_j\frac{|\Vert\bar{\eta}_i\Vert^{2}+2\theta_i|^{\frac{1}{2}}}{|\Vert\bar{\eta}_r\Vert^{2}+2\theta_r|^{\frac{3}{2}}}\frac{\partial}{\partial x_i}|\Vert\bar{\eta}_r\Vert^{2}+2\theta_r|,$$ thus $\frac{\partial}{\partial x_i}|\Vert\bar{\eta}_r\Vert^{2}+2\theta_r|=0$ in $M_i \times M_r.$ Since the coordinates fixed in $M_\ell$ for all $1 \le \ell \le k$ with $\ell \ne i,r$ are arbitrary, it follows that $\frac{\partial}{\partial x_i}|\Vert\bar{\eta}_r\Vert^{2}+2\theta_r|=0$ in $\prod_{s=1}^{k} M_s$. However, $i$ and $r$ are arbitrarily chosen in $\{1,\ldots, t\}$ and $\{t+1, \ldots,k\}$, then \begin{equation}\label{eq1}
\frac{\partial}{\partial x_i}|\Vert\bar{\eta}_r\Vert^{2}+2\theta_r|=0
\end{equation} in $\prod_{s=1}^{k}M_s$ for all $1 \le i \le t$ and all $t+1 \le r \le k$.

By Lemma \ref{geodesic}, \begin{equation}\label{eq2}\frac{\partial}{\partial x_i}|\Vert\bar{\eta}_r\Vert^{2}+2\theta_r|=0 \quad \mbox{in} \quad \prod_{s=1}^{k}M_s\end{equation} for all $t+1 \le r \le k$ and all $m_{t+1}+\ldots+m_{r-1}+t+1 \le i \le m_{t+1}+\ldots+m_r+t$.

Now we prove that $\frac{\partial}{\partial x_i}|\Vert\bar{\eta}_r\Vert^{2}+2\theta_r|=0$ in $\prod_{s=1}^{k}M_s$ for all $t+1 \le r \ne j \le k$ and all $m_{t+1}+\ldots + m_{j-1}+t+1 \le i \le m_{t+1}+\ldots+m_j+t$. In fact, choose $ t+1\le j \ne r\le k$ and fix arbitrary coordinates in $M_{\ell}$ for all $\ell \ne r,j$, then $M_j \times M_r$ is a totally geodesic submanifold of $\prod_{s=1}^{k}M_s$ with induced metric $$\langle \cdot, \cdot \rangle_{jr}^{*}=|\Vert\bar{\eta}_j\Vert^{2}+2\theta_j|^{-1}\sum_{\ell=m_{t+1}+\ldots+m_{j-1}+t+1}^{m_{t+1}+\ldots+m_{j}+t}V_\ell^{2}dx_{\ell}^{2}+|\Vert\bar{\eta}_r\Vert^{2}+2\theta_r|^{-1}\sum_{\ell=m_{t+1}+\ldots+m_{r-1}+t+1}^{m_{t+1}+\ldots +m_{r}+t}V_\ell^{2}dx_{\ell}^{2}.$$

Given arbitrary choices of $m_{t+1}+\ldots+m_{j-1}+t+1 \le i \le m_{t+1}+\ldots+m_{j}+t$, fix the coordinates $\bar{X}_{i}=(x_{m_{t+1}+\ldots+m_{j-1}+t+1}, \ldots, x_{i-1},x_{i+1}, \ldots, x_{m_{t+1}+\ldots+m_{j}+t})$ in $M_j$. Then there exists an open interval $I^{\bar{X}_i}_j$ such that $I^{\bar{X}_{i}}_j\times M_r$ is a totally geodesic submanifold of the product $M_j\times M_r$ whose metric is given by $$\langle \cdot, \cdot \rangle_{i r}^{*}=|\Vert\bar{\eta}_j\Vert^{2}+2\theta_j|^{-1}V_i^{2}dx_i^{2}+|\Vert\bar{\eta}_r\Vert^{2}+2\theta_r|^{-1}\sum_{\ell=m_{t+1}+\ldots+m_{r-1}+t+1}^{m_{t+1}+\ldots+m_r+t}V_\ell^{2}dx_\ell^{2}.$$
Once again, this metric is warped. Repeating the procedure, one can show $\frac{\partial}{\partial x_i}|\Vert\bar{\eta}_r\Vert^{2}+2\theta_r|=0$ in $I_{j}^{\bar{X}_i}\times M_r.$ However, for each $m_{t+1}+\ldots+m_{j-1}+t+1 \le i \le m_{t+1}+\ldots+m_j+t$ we have $M_j=\cup_{{\bar{X}_i}}I_j^{\bar{X}_i}$, therefore $\frac{\partial}{\partial x_i}|\Vert\bar{\eta}_r\Vert^{2}+2\theta_r|=0$ in $M_j \times M_r$. Actually, given the arbitrariness of the choices of $j$ and $r$, we conclude that \begin{equation}\label{eq3}
    \frac{\partial}{\partial x_i}|\Vert\bar{\eta}_r\Vert^{2}+2\theta_r|=0
\end{equation} in $\prod_{s=1}^{k}M_s$ for all $t+1 \le r \ne j \le k$ and all $m_{t+1}+\ldots+m_{j-1}+t+1 \le i \le m_{t+1}+\ldots + m_j+t.$
 Therefore, from equations \eqref{eq1}, \eqref{eq2} and \eqref{eq3}, we see that the functions $\Vert\bar{\eta}_r\Vert^{2}+2\theta_r$, $t+1 \le r \le k,$  are constant in $\prod_{s=1}^{k}M_s$.

 By Lemma \ref{condition} and the conformal Gauss equation, the scalar curvature $s^{*}$ of the Moebius metric is a constant function, because it is a sum of the functions $\Vert\bar{\eta}_r\Vert^{2}+2\theta_r$, $t+1 \le r \le k.$

% \textbf{Second subcase}: Assume that all eigenbundle distributions $E_{\eta_1}, \ldots, E_{\eta_k}$ are of dimension at least two. Once again, applying the same procedure to \textbf{the first subcase}, one can show that all functions $||\bar{\eta}_i||^{2}+2\theta_i$ are constant on $\prod_{j=1}^{k}M_j$. Therefore, $s^{*}$ is constant on $\prod_{j=1}^{k}M_j.$

\textbf{Case (ii)}. \label{caseii} Let $E_{\eta_1}, \ldots, E_{\eta_t}$ be line eigenbundles, and let $E_{\eta_{t+1}}, \ldots, E_{\eta_k}$ be eigenbundles with dimension at least two. (When $t=0$ then all eigenbundles have dimension at least two). Without loss of generality, we have two subcases to consider in $U.$

\begin{description}
    \item[a)] $\Vert\bar{\eta}_{t+1}\Vert^{2}+2\theta_{t+1}=0$ and $\Vert\bar{\eta}_i\Vert^{2}+2\theta_i >0$ for all $1 \le i \ne t+1 \le k.$
    \item[b)] $\Vert \bar{\eta}_1\Vert^{2}+2\theta_1=0$ and $\Vert\bar{\eta}_i\Vert^{2}+2\theta_i > 0$ for all $2 \le i \le k.$ 
\end{description}

\textbf{Subcase \textbf{a)}}. Then Lemma \ref{umbilical} implies that $E_{\eta_i}$ is totally umbilical with mean curvature vector field $-(\grad^{*}(\log (\Vert\bar{\eta}_i\Vert^{2}+2\theta_i)^{-\frac{1}{2}}))_{E_{\eta_i}^{\perp}}$ for all $1 \le i \ne t+1 \le k.$ In addition $E_{\eta_{t+1}}$ is umbilical with respect to $\langle \cdot, \cdot\rangle^{*}$, because $\dim E_{\eta_{t+1}} \geq 2.$ Lemma \ref{integrability} implies that $E_{\eta_i}^{\perp}$ is integrable for all $1 \le i \le k.$ In particular, $(E_{\eta_i})_{i=1}^{k}$ is a twisted product net on $(M^{n}, \langle \cdot, \cdot \rangle^{*})$. By Theorem \ref{decomposition}, $M^{n}$ is locally a twisted product metric \begin{equation}
	\langle \,\,,\,\rangle^{*}=\sum_{i=1}^{k}\rho_i^{2} \pi_i^{*} \langle\,\,,\,\rangle_{i}
\end{equation} in $\prod_{i=1}^{k}M_i$, for some positive \emph{twisting functions} $\rho_i \in C^{\infty}\left(\prod_{i=1}^{k}M_i\right)$, $1 \le i \le k.$ Argumenting as before, one can show that there exist locally orthogonal coordinates $(x_1, \ldots,x_n)$ in $\prod_{i=1}^{k}M_i$ such that the Moebius metric can be written as \begin{align*}\langle \cdot, \cdot \rangle^{*}=\sum_{i=1}^{t}(\Vert\bar{\eta}_i\Vert^{2}+2\theta_i)^{-1}dx_i^{2}&+\sum_{i=t+2}^{k} \sum_{j=m_{t+1}+\ldots+m_{i-1}+t+1}^{m_{t+1}+\ldots+m_{i}+t}(\Vert\bar{\eta}_i\Vert^{2}+2\theta_i)^{-1}V_j^{2}dx_j^{2}\\&+\varphi^{2}\sum_{j=t+1}^{m_{t+1}+t}V_j^{2}dx_j^{2}\end{align*} where $\varphi \in C^{\infty}(\prod_{i=1}^{k}M_i)$ and $V_j=\sqrt{\langle \frac{\partial}{\partial x_j}, \frac{\partial}{\partial x_j}\rangle_i}$ for all $1 \le i \le k$ and $\frac{\partial}{\partial x_j}\in E_{\eta_i}$. Repeating the proof of {case (i)}, one can show that the functions $\Vert\bar{\eta}_i\Vert^{2}+2\theta_i$, $t+2 \le i \le k$, are constants in $\prod_{i=1}^{k}M_i.$ By Lemma \ref{condition} and the conformal Gauss equation, we obtain that the Moebius scalar curvature $s^{*}$ of $f$ is constant in $\prod_{i=1}^{k}M_i.$

\textbf{Subcase b)}. The Lemma \ref{umbilical} shows that $E_{\eta_i}$ is totally umbilical with mean curvature vector field $-(\grad^{*}(\log(\Vert\bar{\eta}_i\Vert^{2}+2\theta_i)^{-\frac{1}{2}}))_{E_{\eta_i}^{\perp}}$ for all $2 \le i \le k$, and $E_{\eta_1}$ is integrable because it is one-dimensional. Furthermore, the Lemma \ref{integrability} shows that $E_{i}^{\perp}$ is integrable for all $1 \le i \le k.$ It is well-known that $f$ is locally holonomic, that is, each point $x\in M^{n}$ lies in an open neighborhood $U\subset M^{n}$ where there exists an orthogonal coordinate system $(x_1, \ldots, x_n)$ such that $E_{\eta_i}=\spa\left\{\frac{\partial}{\partial x_{m_1+\ldots+m_{i-1}+1}}, \ldots, \frac{\partial}{\partial x_{m_1+\ldots+m_i}}\right\}$ for all $1 \le i \le k.$ The induced metric by $f$ is $ds^{2}=\sum_{i=1}^{n}\Vert\partial/\partial x_i\Vert^{2}dx_i^{2}$. For each $x_1$ fixed in $U$, consider the $(n-1)$-slice $(x_2, \ldots, x_n)$, which we denote by $S_{x_1}.$ Then $S_{x_1}$ is an embedded totally geodesic hypersurface of $U$ with induced metric $\sum_{i=2}^{n}\Vert{\partial/\partial x_i\Vert^{2}}dx_i^{2}.$ The Moebius metric induced in $S_{x_1}$ is $\langle \,\,,\,\rangle^{*}_{x_1}=\sum_{i=2}^{n}\rho^{2}\Vert\partial/\partial x_i\Vert^{2}dx_i^{2}.$ Since $E_{\eta_i}$ for $2 \le i \le k$ is totally umbilical with integrable connulity there exists a local product representation $\Phi:\prod_{i=2}^{k} M_i \to S_{x_1}$ that is an isometry with respect to a twisted product metric $$\langle \,\,,\,\rangle_{x_1}=\sum_{i=2}^{k}\rho_{i}^{2}\pi_{i}^{*}\langle \,\,,\,\rangle_{i}$$ on $\prod_{i=2}^{k}M_i.$ Argumenting as case (i), one can prove that $$\langle \,\,,\,\rangle^{*}_{x_1}=\sum_{i=2}^{t}(\Vert\bar{\eta}_i\Vert^{2}+2\theta_i)^{-1}dx_i^{2}+\sum_{i=t+1}^{k} \sum_{j=m_{t+1}+\ldots+m_{i-1}+t+1}^{m_{t+1}+\ldots+m_{i}+t}(\Vert\bar{\eta}_i\Vert^{2}+2\theta_i)^{-1}V_j^{2}dx_j^{2},$$ where $V_j=\sqrt{\langle \frac{\partial}{\partial x_j}, \frac{\partial}{\partial x_j}\rangle_i}$ for all $2 \le i \le k$ and $\frac{\partial}{\partial x_j}\in E_{\eta_i}$. Since this formula holds for each $x_1$ fixed in $U$, we conclude that in $U$ the Moebius metric is written as $$\langle \cdot, \cdot \rangle^{*}=v_1^{2}dx_1^{2}+\sum_{i=2}^{t}(\Vert\bar{\eta}_i\Vert^{2}+2\theta_i)^{-1}dx_i^{2}+\sum_{i=t+1}^{k} \sum_{j=m_{t+1}+\ldots+m_{i-1}+t+1}^{m_{t+1}+\ldots+m_{i}+t}(\Vert\bar{\eta}_i\Vert^{2}+2\theta_i)^{-1}V_j^{2}dx_j^{2}$$ where $v_1=\rho\Vert\partial/\partial x_1\Vert$. Once again, following similar steps to the case (i), we can prove that $\Vert\bar{\eta}_i\Vert^{2}+2\theta_i$ is constant in $U$ for all $t+1 \le i \le k$. Using the conformal Gauss equation and the Lemma \ref{condition} we obtain that the Moebius scalar curvature is constant in $U.$

	\noindent Universidade de S\~ao Paulo\\
	Instituto de Ci\^encias Matem\'aticas e de Computa\c c\~ao\\
	Av. Trabalhador S\~ao Carlense 400\\
	13566-590 -- S\~ao Carlos\\
	BRAZIL \vspace{5pt}\\
    %\noindent Universidade Federal de São Carlos\\
    %Departamento de Matemática\\
	%Rod. Washington Luís km 235 - SP-310\\
	%13565-905 -- S\~ao Carlos \\
	%BRAZIL\\
	\texttt{mateusrodrigues@alumni.usp.br} 
	
\end{document}